\documentclass[9 pt, twocolumn]{IEEEtran}
\usepackage{amsmath,amssymb,euscript,yfonts,psfrag,latexsym,dsfont,graphicx}
\usepackage{bbm,color,amstext,wasysym,parskip,balance}
\usepackage{amsthm}
\usepackage{caption}
\usepackage{subcaption}
\graphicspath{{./},{./figures/}}

\usepackage{algorithm}
\usepackage{algorithmic}
\usepackage{amsmath}

\newtheorem{thm}{Theorem}

\newtheorem{lemma}[thm]{Lemma}
\newtheorem{prop}{Proposition}

\newtheorem{remark}[thm]{Remark}

\newcommand{\mE}{{\mathbb E}}

\newcommand{\mR}{{\mathbb R}}

\newcommand{\cN}{{\mathcal N}}
\newcommand{\cO}{{\mathcal O}}
\newcommand{\cP}{{\mathcal P}}

\newcommand{\cY}{{\mathcal Y}}

\newcommand{\HH}{{\mathrm H}}

\newcommand{\tr}{\operatorname{Tr}}

\newcommand{\argmin}{\operatorname{argmin}}

\newcommand{\red}{\color{black}}

\definecolor{grey}{rgb}{0.6,0.6,0.6}
\definecolor{lightgray}{rgb}{0.97,.99,0.99}

\newlength{\bibitemsep}
\newlength{\bibparskip}
\let\oldthebibliography\thebibliography
\renewcommand\thebibliography[1]{%
  \oldthebibliography{#1}%
  \setlength{\parskip}{\bibitemsep}%
  \setlength{\itemsep}{\bibparskip}%
}

\begin{document}
\title{Covariance Steering for Nonlinear Control-affine Systems}
\author{Hongzhe Yu, Zhenyang Chen, and Yongxin Chen
    \thanks{This work was supported by the NSF under grant 1942523 and 2008513.}
\thanks{H. Yu, Z. Chen, and Y. Chen are with the Georgia Institute of Technology, Atlanta, GA, USA. {\tt\small \{hyu419, zchen927, yongchen\}@gatech.edu}}}

\maketitle

\begin{abstract}
We consider the covariance steering problem for nonlinear control-affine systems. Our objective is to find an optimal control strategy to steer the state of a system from an initial distribution to a target one whose mean and covariance are given. Due to the nonlinearity, the existing techniques for linear covariance steering problems are not directly applicable. By leveraging the Girsanov theorem, we formulate the problem into an optimization over the space of path distributions. We then adopt a generalized proximal gradient algorithm to solve this optimization, where each update requires solving a linear covariance steering problem. Our algorithm is guaranteed to converge to a local solution with a sublinear rate. In addition, each iteration of the algorithm can be achieved in closed form, and thus the computational complexity of it is insensitive to the resolution of time-discretization. {\red In the examples, our method achieves 1000 times speedup over an existing algorithm.}  
\end{abstract}

\section{Introduction}

%
%
%
%

Uncertainties are ubiquitous in engineering systems and an important task of feedback control is to mitigate their effects on the behavior of the systems. Very often the goal of control is to ensure the system does not deviate too much from the desired behavior in the presence of uncertainties. A standard approach for this task is optimal control where a cost function is used to promote the desired system behavior. The cost function is a design parameter and is often chosen by trial and error until the desired performance is achieved. 
The covariance control/assignment paradigm was born \cite{HotSke87,IwaSke92,XuSke92,GriSke97} in the quest of providing a more direct approach to mitigate uncertainties. In its original formulation, the goal was to find a control strategy for a linear time-invariant stochastic system so that its state has a specified covariance in steady state. The covariance control theory was recently extended to the finite horizon control setting \cite{CheGeoPav15a,CheGeoPav15b,CheGeoPav15c,HalWen16,CheGeoPav17a,Bak16a,Bak18} where one seeks to steer the state covariance of a linear dynamic system from an initial value to a target value. The covariance control framework has been successfully used in a range of applications \cite{ZhuGriSke95,RidTsi18,OkaTsi19,Kal02}.

Most existing works on covariance control/steering are for linear dynamics. The methods developed in them are not directly applicable to nonlinear problems. The purpose of this work is to develop a covariance steering method for nonlinear stochastic dynamics. 
Even though for nonlinear stochastic systems the state distribution is no longer Gaussian, the covariance still provides valuable information about the uncertainties of the state variables. Indeed, under the mild assumption that the state follows a sub-Gaussian distribution, the covariance can properly capture the high probability region and thus the uncertainty of the state \cite{RudVer13}. The same idea of approximating the full distribution with first order and second order moments has been widely used in filtering algorithms such as the extended Kalman filter algorithm.

In this paper, we develop an efficient algorithm of nonlinear covariance steering for control-affine nonlinear systems. 
Our method is based on the celebrated Girsanov theorem \cite{Gir60} for stochastic processes, which connects the control energy with the Kullback-Leibler divergence \cite{Cov99} between the measures induced by the controlled and uncontrolled dynamics. As a result, the covariance steering problem we consider can be reformulated as an infinite dimensional optimization problem over the space of measures of trajectories. The optimization can then be solved using the generalized proximal gradient descent algorithm \cite{BecTeb03,Bec17}, a popular algorithm for nonlinear programming. For our covariance steering problem, the proximal gradient descent algorithm converts the problem into a sequence of linear covariance steering problems, and each of these linear covariance steering problems has a closed-form solution. Our algorithm is guaranteed to converge to a local optimal solution with sublinear rate $\cO(1/k)$. Moreover, each iteration of the algorithm is associated with a feasible solution to the problem. Thus, whenever we stop the algorithm, we can still obtain a controller that steer the state to desired target mean and covariance.

The nonlinear covariance control problem was recently considered in \cite{RidOkaTsi19,YiCaoTheChe20,BakTso20,TsoBak20}. In \cite{RidOkaTsi19} the cost function is assumed to be quadratic and an iterative linearization method resembling sequential quadratic programing is used. The algorithm is not guaranteed to converge. In \cite{YiCaoTheChe20} a method based on dynamic differential programming was developed. To enforce the covariance constraint on the terminal state, the algorithm first relaxes the constraint using Lagrangian method and then uses primal-dual updates to search for the solution. In \cite{BakTso20}, nonlinear covariance steering for discrete dynamics was studied. The method is similar to that in \cite{RidOkaTsi19} but an unscented transform is used to better capture the nonlinearity of the dynamics. The method was extended to the data-driven setting in \cite{TsoBak20} where a Gaussian process is used to identify the unknown dynamics. Compared with all these works, our method is tailored for continuous-time control-affine dynamics with general state cost. A distinguishing feature of the proposed algorithm is that its complexity is insensitive to the resolution of time-discretization in implementation while all the other algorithms do not scale well with this resolution. {\red In our experiments, our method shows at least 1000 times acceleration compared with \cite{RidOkaTsi19}.}

The rest of the paper is structured as follows. In Section \ref{sec:back} we provide the background on covariance control for linear dynamics and the proximal gradient algorithm. The problem is formulated in Section \ref{sec:formulation} and the algorithm is developed in Section \ref{sec:algorithm}. An extension of the method to a slightly more general systems is presented in Section \ref{sec:extension},
which is followed by a numerical example in Section \ref{sec:example} and a concluding remark in Section \ref{sec:conclusion}.

\section{Background}\label{sec:back}
In this section, we provide a brief introduction to the covariance steering problem for continuous-time linear systems \cite{CheGeoPav15a,CheGeoPav15b,CheGeoPav15c}. We also present key steps of the proximal gradient algorithm, on which our method is based.
\subsection{Covariance steering for linear systems}\label{sec:linearcov}
In the covariance steering/control problems for linear systems \cite{CheGeoPav15a}, the objective is to drive the state of a linear stochastic system from an initial Gaussian random vector $X_0\sim \cN(m_0, \Sigma_0)$ at $t=0$ to a terminal one $X_1\sim \cN(m_1, \Sigma_1)$ at $t=1$ over a time interval $[0, 1]$~\footnote{Any finite time interval $[0, T]$ can be converted to $[0, 1]$ by rescaling. Thus, without loss of generality, we use the time interval $[0, 1]$ throughout.}. In \cite{CheGeoPav15a,CheGeoPav15c}, the linear dynamics under consideration is 
	\[
		dX_t = A(t)X_t dt + B(t)(u_tdt+ \sqrt{\epsilon} dW_t),
	\]
which describes the behavior of linear stochastic control system whose actuation is corrupted by white noise. Here the pair $A(t)\in \mR^{n\times n}, B(t) \in \mR^{n\times p}$ is assumed to be controllable, $W_t\in \mR^p$ represents a standard Wiener process \cite{CohEll15}, and $\epsilon>0$ parameterizes the intensity of the disturbance. 
%
When one seeks also to minimize a certain cost function, the covariance steering problem for linear dynamics in the continuous-time setting is \cite{CheGeoPav15c}
	\begin{subequations}\label{eq:covcontrollinear}
	\begin{eqnarray}
		\min_{u} && \mE \left\{\int_0^1 [\frac{1}{2}\|u_t\|^2 + \frac{1}{2} X_t^T Q(t) X_t]dt\right\}
		\\&& dX_t = A(t)X_t dt + B(t)(u_tdt+ \sqrt{\epsilon} dW_t)
		\\&& X_0 \sim \cN(m_0, \Sigma_0),\quad X_1 \sim \cN(m_1, \Sigma_1),
	\end{eqnarray}
	\end{subequations}
where the optimization is over all the state feedback control strategies. The state cost matrix $Q(t)\in \mR^{n \times n}$ is often assumed to be positive semi-definite. This problem differs from standard linear quadratic optimal control problems in that there is no explicit terminal cost. Instead, a constraint $X_1 \sim \cN(m_1, \Sigma_1)$ on the statistics of the terminal state is imposed. 
	
Thanks to linearity, the mean and covariance of the system can be controlled separately \cite{CheGeoPav15a,CheGeoPav15b,CheGeoPav15c}. The control of the mean is the solution to the standard optimal control problem 
	\begin{subequations}
	\begin{eqnarray}
		\min_{v} && \int_0^1 [\frac{1}{2}\|v_t\|^2 + \frac{1}{2} x_t^T Q(t) x_t]dt
		\\&& \dot x_t = A(t)x_t + B(t)v_t
		\\&& x_0 =m_0, \quad x_1 = m_1.
	\end{eqnarray}
	\end{subequations}
The control for the covariance is of the form $-B(t)^T\Pi(t) X_t$ with $\Pi(\cdot), \HH(\cdot)$ satisfying a coupled Riccati equations \cite{CheGeoPav17a}
    \begin{subequations}\label{eq:LQschrodinger}
    \begin{eqnarray}\label{eq:LQschrodinger1}
   	-\dot\Pi(t)\!\!\!\!&=&\!\!\!\!A^T\Pi(t)\!+\!\Pi(t)A(t)\!-\!\Pi(t)B(t)B(t)^T\Pi(t) \!+\!Q(t)
    \\\label{eq:LQschrodinger2}
    -\dot\HH(t)\!\!\!\!&=&\!\!\!\!A^T\HH(t)\!+\!\HH(t)A(t)\!+\!\HH(t)B(t)B(t)^T\HH(t) \!-\!Q(t)
    \\\label{eq:LQschrodinger3}
    \epsilon\Sigma_0^{-1}\!\!\!\!&=&\!\!\!\!\Pi(0)+\HH(0)
    \\\label{eq:LQschrodinger4}
    \epsilon\Sigma_1^{-1}\!\!\!\!&=&\!\!\!\!\Pi(1)+\HH(1).
    \end{eqnarray}
    \end{subequations}
It turns out \eqref{eq:LQschrodinger} has a unique solution as in the following result \cite{CheGeoPav15c}. 
\begin{thm}
The coupled system of Riccati equations (\ref{eq:LQschrodinger1}-\ref{eq:LQschrodinger4}) has a unique solution, which is determined by the
initial value problem consisting of (\ref{eq:LQschrodinger1}-\ref{eq:LQschrodinger2}) and
    \begin{subequations}\label{eq:initial}
\begin{eqnarray}
    \Pi(0)&=& \frac{\epsilon\Sigma_0^{-1}}{2}-\Phi_{12}^{-1}\Phi_{11}
    \\\nonumber &&-\Sigma_0^{-1/2}\left(\frac{\epsilon^2I}{4}+\Sigma_0^{1/2}\Phi_{12}^{-1}\Sigma_1
    (\Phi_{12}^T)^{-1}\Sigma_0^{1/2}\right)^{1/2}\Sigma_0^{-1/2},
    \label{eq:initial_a}
    \\
    \HH(0)&=& \epsilon\Sigma_0^{-1}-\Pi(0),
    \end{eqnarray}
    \end{subequations}
where
    \begin{equation}\nonumber
        \Phi(t,s)=\left[
        \begin{matrix}
        \Phi_{11}(t,s) & \Phi_{12}(t,s)\\
        \Phi_{21}(t,s) & \Phi_{22}(t,s)
        \end{matrix}\right]
    \end{equation}
is the state transition matrix corresponding to $\partial \Phi(t,s)/\partial t = M(t)\Phi(t,s)$ with $\Phi(s,s)=I$ and
    \[
       M(t)= \left[
        \begin{matrix}A(t) & -B(t)B(t)^T\\-Q(t)& -A(t)^T\end{matrix}
        \right],
    \]
and 
    \[
        \left[
        \begin{matrix}
        \Phi_{11}& \Phi_{12}\\
        \Phi_{21}& \Phi_{22}
        \end{matrix}\right]
        :=\left[
        \begin{matrix}
        \Phi_{11}(1,0) & \Phi_{12}(1,0)\\
        \Phi_{21}(1,0) & \Phi_{22}(1,0)
        \end{matrix}\right].
    \]
\end{thm}

Denote the optimal control for mean to be $v_t^\star$ and corresponding state trajectory to be $x_t^\star$. Combining it with the covariance control yields the optimal feedback policy
	\[
		u_t^\star = -B(t)^T\Pi(t) (X_t-x_t^\star) + v_t^\star.
	\]

\subsection{Proximal gradient algorithm}
The proximal gradient algorithm \cite{Bec17} is designed for the composite optimization
	\begin{equation}\label{eq:composite}
		\min_{y\in\cY} F(y) + G(y),
	\end{equation}
where $\cY$ denotes a feasible set. The function $F$ is assumed to be smooth. The function $G$ is usually a regularizer that is often not smooth. The algorithm follows the update
	\begin{subequations}\label{eq:pgupdate}
	\begin{eqnarray}\label{eq:pgupdate1}
		y^{k+1}\!\!\!\! &=&\!\!\!\!\! \argmin_{y\in \cY} G(y) + \frac{1}{2 \eta} \|y-(y^k - \eta \nabla F(y^k))\|^2
		\\\!\!\!\!&=&\!\!\!\!\! \argmin_{y\in \cY} G(y) \!+\! \frac{1}{2 \eta} \|y\!-\!y^k\|^2\!\! +\!\! \langle \nabla F(y^k), y\!-\!y^k\rangle \label{eq:pgupdate2}
	\end{eqnarray}
	\end{subequations}
with $\eta>0$ being the stepsize. One advantage of the proximal gradient algorithm is that it only evaluates the gradient of $F$ and doesn't require even the differentiability of $G$. In many applications, $G$ is a regularizer of simple form, e.g., 1-norm, nuclear norm, and the minimization \eqref{eq:pgupdate} can be carried out efficiently. 

The proximal gradient algorithm has been generalized to the non-Euclidean setting. It is built upon the mirror descent method \cite{BecTeb03,Bec17}. Let $D(\cdot, \cdot)$ be a Bregman divergence, then the generalized non-Euclidean proximal gradient algorithm reads
	\begin{equation}\label{eq:gpg}
		y^{k+1} = \argmin_{y\in \cY} G(y) + \frac{1}{\eta} D(y, y^k) + \langle \nabla F(y^k), y-y^k\rangle.
	\end{equation}
A popular choice of $D(\cdot,\cdot)$ is the Kullback-Leibler divergence ${\rm KL}(\cdot \| \cdot)$, which is suitable for optimization over probability vectors/distributions. 

The proximal gradient algorithm enjoys nice convergence properties. When both $F$ and $G$ are convex, the algorithm is guaranteed to converge to the global minimum with rate $\cO(1/k)$ \cite{BecTeb03,Bec17}. When $F$ is nonconvex, one can only expect for convergence to local solutions. It turns out that objective function $F(y)+G(y)$ is monotonically decreasing along the updates, and the updates converge to some stationary points with sublinear rate $\cO(1/k)$ with respect to some suitable criteria \cite{LiZhoLiaVar17}. 

\section{Problem formulation}\label{sec:formulation}
Consider a nonlinear control-affine system
	\begin{equation}\label{eq:nonlinear}
		dX_t = f(t,X_t)dt + B(t) (u_t dt + \sqrt{\epsilon} d W_t)
	\end{equation}
where $X_t\in \mR^n$ represents the state, and $W_t$ denotes a standard Wiener process. The drift function $f(t,x)$ is assumed to be locally Lipschitz continuous with respect to $x$ and continuous with respect to $t$. The input matrix $B(t)\in\mR^{n\times p}$ is assumed to be continuous and full rank. Many nonlinear dynamics in engineering applications can be modeled by \eqref{eq:nonlinear} where $\sqrt{\epsilon} d W_t$ captures the uncertainties in the actuation. Note that the noise is assumed to enter the dynamics in the same channel as the input. This is a widely used assumption in stochastic control in both theory and practice \cite{CheGeoPav14e,ThiKap15,WilThe17,CalHal21}.

We are interested in the covariance control problems of steering the state statistics of the system from an initial value to a target one. In particular, we seek a state feedback control strategy that achieves this goal and meanwhile minimizes a cost function. The covariance steering problem we consider is
	\begin{subequations}\label{eq:covcontrol}
	\begin{eqnarray}\label{eq:covcontrol1}
		\min_{u} && \mE \left\{\int_0^1 [\frac{1}{2}\|u_t\|^2 + V(X_t)]dt\right\}
		\\\label{eq:covcontrol2}
		&&dX_t = f(t,X_t)dt + B(t) (u_t dt + \sqrt{\epsilon} d W_t)
		\\\label{eq:covcontrol3}
		&& 
		X_0 \sim \rho_0,\quad X_1 \sim \rho_1,
	\end{eqnarray}
	\end{subequations}
where $\rho_0$ ($\rho_1$) is a probability distribution with mean $m_0$ ($m_1$) and covariance $\Sigma_0$ ($\Sigma_1$). The cost function has two decoupled terms: $\frac{1}{2}\|u_t\|^2$ that only depends on the control and $V(X_t)$ that only depends on the state. The problem \eqref{eq:covcontrol} has been investigated in the study of distribution control \cite{CheGeoPav14e,CheGeoPav14d,CalHal20,CalHal21} for general distributions $\rho_0, \rho_1$, which links control theory and optimal transport theory \cite{CheGeoPav14e,CheGeoPav20}.
	
The special structure of the cost and dynamics in \eqref{eq:covcontrol} leads to an elegant reformulation as an optimization over probability measures. Denote by $\cP^u$ the distribution over the path space $\Omega:=C([0,1], \mR^n)$ induced by the stochastic process \eqref{eq:nonlinear}, and as a convention, by $\cP^0$ the distribution induced by the same process with zero control, then {\red the Girsanov theorem \cite{Gir60,IkeWat14} states that 
    \begin{subequations}\label{eq:girsanov}
    \begin{equation}\label{eq:girsanova}
        \frac{d\cP^u}{d\cP^0} = \exp \left(\int_0^1\frac{1}{2\epsilon}\|u_t\|^2dt + \frac{1}{\sqrt\epsilon}u_t^T dW_t\right).
    \end{equation}
When $B$ is full (column) rank, this can be equivalently written as
    \begin{equation}\label{eq:girsanovb}
        \frac{d\cP^u}{d\cP^0} = \exp \left(\int_0^1\frac{1}{2\epsilon}\|B u_t\|^2_{(BB^T)^\dagger} dt + \frac{1}{\sqrt\epsilon}u_t^T dW_t\right).
    \end{equation}
    \end{subequations}
It follows from the Girsanov theorem that \cite{Ben71,Dai91,CheGeoPav14e}}
	\begin{equation}
		\mE \left\{\int_0^1\frac{1}{2\epsilon}\|u_t\|^2dt\right\}={\rm KL} (\cP^u\| \cP^0):= \int \log\frac{d\cP^u}{d\cP^0} d\cP^u.
	\end{equation}
Intuitively, it says that the difference between the controlled and uncontrolled processes can be quantified by the control energy. This relation 
is heavily utilized recently in the distributional control problem \cite{CheGeoPav14e}.
Thus, the problem \eqref{eq:covcontrol} can be reformulated as
	\begin{subequations}\label{eq:KLformulation}
	\begin{eqnarray}\label{eq:KLformulation1}
		\min_{\cP^u} &&\int d\cP^u \left[\log\frac{d\cP^u}{d\cP^0} +\frac{1}{\epsilon}V\right]
		\\\label{eq:KLformulation2}
		&& (X_0)_\sharp \cP^u = \rho_0, \quad (X_1)_\sharp \cP^u = \rho_1.
	\end{eqnarray}
	\end{subequations}
Here $(X_0)_\sharp \cP^u$ ($(X_1)_\sharp \cP^u$) stands for the distribution of $X_0$ ($X_1$) when the process $X_t$ is associated with the distribution $\cP^u$. 

Denote by $\Pi (\rho_0, \rho_1)$ the set of all distributions over the path space $\Omega$ such that the constraints \eqref{eq:KLformulation2} hold. Let
	\[
		F(\cP^u) = \int [\frac{1}{\epsilon}V-\log d\cP^0] d\cP^u
	\]
and
	\begin{equation}\label{eq:entropy}
		G(\cP^u) = \int d\cP^u \log d\cP^u.
	\end{equation}
Then, \eqref{eq:KLformulation} becomes a composite optimization
	\begin{equation}\label{eq:FG}
		\min_{\cP^u\in\Pi(\rho_0, \rho_1)} F(\cP^u) + G(\cP^u).
	\end{equation}
This is an infinite dimensional optimization. For general marginals $\rho_0, \rho_1$ and nonlinear dynamics \eqref{eq:nonlinear}, the problem can be addressed using the distribution control framework \cite{CheGeoPav14e,CalHal21}. However, this method requires solving a coupled partial differential equation system and does not scale to cases with large state dimension $n$ \cite{CheGeoPav14e,CalHal21}.

{\red Our goal is to develop an efficient algorithm to approximately solve \eqref{eq:FG}.} In the covariance steering formulation, instead of solving \eqref{eq:FG} exactly, we look for an approximate solution $\cP^u$ that is induced by a Gaussian Markov process. Thus, we restrict our search space to $\hat \Pi(\rho_0,\rho_1)$, the space of measures over the path space $\Omega$ that are induced by Gaussian Markov processes and have marginal distributions $\cN(m_0, \Sigma_0), \cN(m_1,\Sigma_1)$. When $\cP$ has relatively small variance, $F(\cP)$ can be approximated by	
	\[
		\int [\frac{1}{\epsilon}\hat V-\log d\hat \cP^0] d\cP
	\]
where $\hat V$ is the second order approximation of $V$ and $\hat\cP^0$ is a Gaussian Markov approximation of $\cP^0$, both along the mean of $\cP$. More precisely, let the mean of $\cP$ be the trajectory $z_t$,
then 
	\begin{equation}\label{eq:linearizeV}
		\hat V(t,x) = V(z_t) + (x^T-z_t^T) \nabla V(z_t) + \frac{1}{2}(x^T-z_t^T) \nabla^2 V(z_t) (x-z_t), 
	\end{equation}
and $\hat\cP^0$ is associated with the Gaussian Markov process
	\begin{equation}\label{eq:linearizeP0}
		dX_t = \nabla f(t,z_t)^T X_t dt + [f(t,z_t) - \nabla f(t, z_t)^T z_t]dt+\sqrt{\epsilon} B(t) dW_t.
	\end{equation}

We have thus arrived at the following nonlinear covariance steering problem
	\begin{equation}\label{eq:noncovcontrol}
		\min_{\cP^u\in\hat\Pi(\rho_0, \rho_1)} \int [\frac{1}{\epsilon}\hat V-\log d\hat \cP^0] d\cP^u + \int d\cP^u \log d\cP^u.
	\end{equation}
This nonlinear covariance steering problem turns out to be considerably easier than the distribution control problem \eqref{eq:FG} {\red from a computational point of view.} Indeed, in the following section, we develop an efficient algorithm for it that is scalable to problems with large state dimension.

\section{Proximal gradient algorithm for covariance steering}\label{sec:algorithm}
The nonlinear covariance steering problem \eqref{eq:noncovcontrol} is clearly a composite optimization
	\begin{equation}\label{eq:noncovcontrolopt}
		\min_{\cP^u\in\hat\Pi(\rho_0, \rho_1)} F(\cP^u) + G(\cP^u)
	\end{equation}
where $G$ is as in \eqref{eq:entropy}, and by abuse of notation,
	\begin{equation}\label{eq:FP}
		F(\cP^u)=\int [\frac{1}{\epsilon}\hat V-\log d\hat \cP^0] d\cP^u = \langle \frac{1}{\epsilon}\hat V-\log d\hat \cP^0, \cP^u \rangle.
	\end{equation}
In this section, we develop an efficient algorithm to solve \eqref{eq:noncovcontrolopt} based on the generalized proximal gradient algorithm \eqref{eq:gpg}.
\subsection{Main algorithm}
Since \eqref{eq:noncovcontrolopt} is an optimization over the space of probability measures, we use Kullback-Leibler divergence in \eqref{eq:gpg}. This leads to the following iteration
	\begin{equation}\label{eq:PG}
		\cP_{k+1} = \argmin_{\cP\in\hat\Pi(\rho_0, \rho_1)} G(\cP) + \frac{1}{\eta} {\rm KL}( \cP\| \cP_k)+ \langle \frac{\delta F}{\delta \cP}(\cP_k), \cP\rangle. 
	\end{equation}
Here we use variation $\frac{\delta F}{\delta \cP}$ instead of gradient since $\cP$ is an infinite dimensional object. As we will see below, each iteration can be realized by solving a linear covariance steering problem. 

First, we derive an explicit expression for $\frac{\delta F}{\delta \cP}$. Note that both $\hat V$ and $\hat\cP^0$ depend on the mean of $\cP$, thus $F(\cP)$ is not a linear function of $\cP$. The variation $\frac{\delta F}{\delta \cP}$ thus has some extra terms other than $\frac{1}{\epsilon}\hat V-\log d\hat \cP^0$. 
\begin{lemma}\label{lem:variationP}
Let $\cP$ be the path distribution induced by a Gaussian Markov process
	\[
		dX_t = A(t) X_t dt + a(t) dt + \sqrt{\epsilon} B(t) dW_t,
	\]
then the variation of $F(\cP)$ with respect to $\cP$ is
{\red
	\begin{eqnarray}\label{eq:gradietnF}
		\frac{\delta F}{\delta \cP}(\cP)(t, x) &=& \frac{1}{\epsilon}\hat V -\log d\hat \cP^0 + \frac{1}{2\epsilon} x^T \nabla\tr(\nabla^2 V(z_t) \Sigma_t) 
		\\\nonumber &&\hspace{-1.4cm}+ \frac{1}{2\epsilon} x^T \nabla \tr((BB^T)^{\dagger}(\nabla f(z_t)^T- A) \Sigma_t (\nabla f(z_t)-A^T)),
	\end{eqnarray}
where $(\cdot)^\dagger$ represents pseudo-inverse, $z_t$ and $\Sigma_t$} are the mean and covariance of $\cP$, respectively, and $\hat V, \hat\cP^0$ are as in \eqref{eq:linearizeV} and \eqref{eq:linearizeP0}.
\end{lemma}

We remark that the extra terms involve the second order derivative of the drift $f$ and the third order derivative of $V$. Thus, if the dynamics is linear and the cost is quadratic as in the linear covariance steering problems, $\frac{\delta F}{\delta \cP}(\cP) = \frac{1}{\epsilon}\hat V -\log d\hat \cP^0$ and is independent of $\cP$. We also remark that $\frac{\delta F}{\delta \cP}$ in \eqref{eq:gradietnF} is a quadratic function. 
	
We next establish that each iteration of \eqref{eq:PG} is a covariance control problem for linear dynamics. To this end, we represent $\cP^u_k$ by a Gaussian Markov process
	\begin{equation}\label{eq:Pk}
		dX_t = A_k(t) X_t dt + a_k(t) dt + \sqrt{\epsilon} B(t) dW_t.
	\end{equation}
Denote its mean and covariance by $z^k$ and $\Sigma^k$ respectively, then 
	\begin{subequations}
	\begin{eqnarray}
		\dot z^k_t &=& A_k(t)z^k_t +a_k(t)
		\\
		\dot \Sigma_t^k &=& A_k(t)\Sigma_t^k +\Sigma_t^k A_k(t)^T + \epsilon B(t)B(t)^T.
	\end{eqnarray}
	\end{subequations}	
Plugging \eqref{eq:gradietnF} into \eqref{eq:PG}, in view of representation \eqref{eq:Pk} of $\cP^u_k$, we obtain
	\begin{eqnarray}\nonumber
		\cP^u_{k+1} \!\!\!&=&\!\!\!\!\! \argmin_{\cP\in \hat\Pi(\rho_0, \rho_1)} \int [\frac{1}{\epsilon}\hat V+ \frac{1}{2\epsilon} x^T \nabla\tr(\nabla^2 V(z_t^k) \Sigma_t^k) 
		\\\!\!\!&+&\!\!\!\!\! \frac{1}{2\epsilon} x^T \nabla \tr((BB^T)^{\dagger}(\nabla f(z_t^k)^T- A_k) \Sigma_t^k (\nabla f(z_t^k)-A_k^T))\nonumber
		\\\!\!\!&-&\!\!\!\!\!\log d\hat\cP^0\!-\!\frac{1}{\eta} \log d\cP^u_k] d\cP\! +\! (1\!+\!\frac{1}{\eta})\!\!\int d\cP \log d\cP.\label{eq:iteration}
	\end{eqnarray}
	
Linearizing the uncontrolled system 
	\begin{equation}\label{eq:nonlinearzero}
		dX_t = f(t,X_t)dt + \sqrt{\epsilon}  B(t)d W_t
	\end{equation}
along $z^k$ yields the linear approximation
	\begin{equation}\label{eq:GM}
		dX_t = \hat A_k(t) X_t dt + \hat a_k(t) dt + \sqrt{\epsilon} B(t) dW_t
	\end{equation}
with
	\begin{subequations}\label{eq:zS}
	\begin{eqnarray}
		\hat A_k(t) &=& \nabla f(t, z^k_t)^T
		\\
		\hat a_k(t) &=& f(t,z^k_t) - \nabla f(t, z^k_t)^T z^k_t.
	\end{eqnarray}
	\end{subequations}
This linear approximation corresponds to $\hat\cP^0$ at the $k$-th iteration. This is the same as \eqref{eq:linearizeP0} with some additional notations being introduced to emphasize the dependency on the iteration number $k$. With these Gaussian Markov representations of $\cP^0, \cP_k^u$ we establish the following.
\begin{thm}\label{thm:eachiteration}
Each proximal gradient iteration \eqref{eq:iteration} amounts to solving the following linear covariance steering problem
	\begin{subequations}\label{eq:iterlinear}
	\begin{eqnarray}\label{eq:iterlinear1}
		\min_{u}\!\!\!\! &&\!\!\!\! \mE \left\{\int_0^1 [\frac{1}{2}\|u_t\|^2 + \frac{1}{2} X_t^T Q_k(t) X_t + X_t^T r_k(t)]dt\right\}
		\\ \label{eq:iterlinear2}\!\!\!\! &&\!\!\!\!  \hspace{-0.1cm}dX_t = \frac{1}{1+\eta}[ A_k(t)+\eta\hat A_k(t)] X_t dt +  \frac{1}{1+\eta}[a_k(t)
		\\\nonumber \!\!\!\! &&\!\!\!\!  +\eta\hat a_k(t)] dt + B(t)(u_tdt+ \sqrt{\epsilon} dW_t)
		\\ \label{eq:iterlinear3}&& X_0 \sim \rho_0,\quad X_1 \sim \rho_1,
	\end{eqnarray}
	\end{subequations}
where
	\begin{eqnarray*}
		Q_k(t) &=& \frac{\eta}{(1+\eta)}\nabla^2 V(z^k_t)+
		\\\nonumber &&\hspace{-1cm}\frac{\eta}{(1+\eta)^2} (A_k(t)-\hat A_k(t))^T(B(t)B(t)^T)^{\dagger}(A_k(t)-\hat A_k(t))
	\end{eqnarray*}
and
	\begin{align*}
		r_k(t) =& \frac{\eta}{1+\eta}\nabla V+\frac{\eta}{2(1+\eta)}[\nabla \tr(\nabla^2V\Sigma^k_t)-2\nabla^2V(z^k_t)z^k_t
		\\&\hspace{-0.5cm}+\nabla\tr( (B(t)B(t)^T)^{\dagger}(\nabla f^T -A_k)\Sigma^k_t(\nabla f-A_k^T))]
		\\&\hspace{-0.5cm} +\frac{\eta}{(1+\eta)^2}(A_k(t)-\hat A_k(t))^T(B(t)B(t)^T)^{\dagger}(a_k(t)-\hat a_k(t)).
	\end{align*}
	
\end{thm}
\begin{proof}
The optimization in \eqref{eq:iteration} is equivalent to 
	\begin{align}\nonumber
		\min_{\cP\in \hat\Pi(\rho_0, \rho_1)} &\int [\frac{1}{\epsilon}\hat V+ \frac{1}{2\epsilon} x^T \nabla\tr(\nabla^2 V(z_t^k) \Sigma_t^k)+ 
		\\\nonumber&\hspace{-1.16cm}\frac{1}{2\epsilon} x^T \nabla \tr((BB^T)^{\dagger}(\nabla f(z_t^k)^T- A_k) \Sigma_t^k (\nabla f(z_t^k)-A_k^T))] d\cP
		\\&+{\rm KL} (\cP\| \hat \cP^0) + \frac{1}{\eta} {\rm KL}(\cP\| \cP_k^u). \label{eq:objKL}
	\end{align}
Let $\cP = \cP^u$ be parametrized by \eqref{eq:iterlinear2}, then, by Girsanov theorem \eqref{eq:girsanov}, ${\rm KL} (\cP^u\| \hat \cP^0)$ equals
	\begin{equation}\label{eq:KLP0}
		 \frac{1}{2\epsilon} \mE\! \left\{\!\int_0^1 \!\!\|\frac{1}{1+\eta}(A_k X_t-\!\hat A_k X_t +a_k-\hat a_k)+B u_t\|^2_{(BB^T)^\dagger} dt\right\}
	\end{equation}
and ${\rm KL} (\cP^u\| \cP^u_k)$ equals
	\begin{equation}\label{eq:KLPk}
	\frac{1}{2\epsilon} \mE \!\left\{\!\int_0^1 \|\frac{\eta}{1+\eta}(\hat A_k X_t-\!A_k X_t +\!\hat a_k- a_k)+Bu_t\|^2_{(BB^T)^\dagger} dt\right\}
	\end{equation}	
Combining \eqref{eq:KLP0} and \eqref{eq:KLPk} points to
	\begin{align*}
		&{\rm KL} (\cP^u\| \hat \cP^0)+\frac{1}{\eta}{\rm KL} (\cP^u\| \cP^u_k)
		\\=& \frac{1}{2\epsilon} \mE \left\{\int_0^1 [(1+\frac{1}{\eta})\|u_t\|^2 +\right.
		\\&\left. \frac{1}{1+\eta}\|( A_k X_t-\hat A_k X_t + a_k- \hat a_k)\|^2_{(B(t)B(t)^T)^\dagger} ]dt\right\}.
	\end{align*}
{\red In view of \eqref{eq:linearizeV}, plugging the above into \eqref{eq:objKL} yields the objective function \eqref{eq:iterlinear1}, after removing some constant terms and rescaling with ratio $(1+\eta)/(\epsilon\eta)$.}
\end{proof}

The linear covariance steering problem \eqref{eq:iterlinear} has a closed-form solution (see Section \ref{sec:lineardrift} for more details). Denote the optimal control policy to \eqref{eq:iterlinear} by
	\begin{equation}\label{eq:optKd}
		u_t^\star = K_k(t) X_t + d_k(t),
	\end{equation}
then $\cP_{k+1}^u$ is induced by the closed-loop process 
	\begin{align*}
		dX_t &= \frac{1}{1+\eta}[A_k(t)+\eta \hat A_k(t)] X_t dt +  \frac{1}{1+\eta}[ a_k(t)+\eta\hat a_k(t)] dt 
		\\&+ B(t)(u_t^\star dt+ \sqrt{\epsilon} dW_t).
	\end{align*}
It follows that
	\begin{subequations}\label{eq:updateAa}
	\begin{eqnarray}\label{eq:updateAa1}
		A_{k+1}(t) &=&  \frac{1}{1+\eta}[ A_k(t)+\eta\hat A_k(t)] + B(t)K_k(t)
		\\\label{eq:updateAa2}
		a_{k+1}(t) &=& \frac{1}{1+\eta}[a_k(t)+\eta\hat a_k(t)] + B(t)d_k(t).
	\end{eqnarray}
	\end{subequations}
The above equations \eqref{eq:updateAa} provide an extremely simple rule to update $\cP^u_k$, based on which we present the proximal gradient algorithm (Algorithm \ref{alg:noncov}) for nonlinear covariance steering.

There are many choices for initialization. For instance, one can set $\cP^u_0$ to be the process $dX_t = \sqrt{\epsilon} B(t)dW_t$, which means $A_0(t)\equiv 0$ and $a_0(t)\equiv 0$. An alternative option is a linearization of the prior process \eqref{eq:nonlinearzero}. More precisely, set $z^0_t$ to be the solution to $\dot z^0_t = f(t, z^0_t)$
and $A_0(t) = \nabla f(t, z^0_t)^T$.
The initial $a_0(t)$ can be calculated by
	\[
		a_0(t) = f(t,z^0_t) - A_0(t)z^0_t.
	\]
\begin{algorithm}[ht]
    \caption{Proximal gradient nonlinear covariance steering algorithm}
    \label{alg:noncov}
 \begin{algorithmic}
    \STATE Initialize $A_0, a_0$
    \FOR{$k = 1,2,\ldots$} 
    \STATE Update $A_k$ using \eqref{eq:updateAa1}
    \STATE Update $a_k$ using \eqref{eq:updateAa2}
     \ENDFOR
 \end{algorithmic}
\end{algorithm}
\begin{remark}
Algorithm \ref{alg:noncov} inherits the convergence properties of the proximal gradient algorithm and thus converges to a local solution to \eqref{eq:noncovcontrol} with sublinear rate $\cO (1/k)$. 
\end{remark}
\begin{remark}
The computational cost of Algorithm \ref{alg:noncov} can be divided into two parts: i) linearization to get $\hat A_k, \hat a_k$ and ii) obtaining optimal strategy \eqref{eq:optKd} for the linear covariance steering problem. Both have closed-form and can be computed efficiently, even for high resolution of time-discretization. This is a major advantage over existing works on nonlinear covariance steering. 
\end{remark}

\subsection{Covariance steering for linear systems with drift}\label{sec:lineardrift}
The linear covariance control problem in Theorem \ref{thm:eachiteration} is slightly more general than \eqref{eq:covcontrollinear} due to some extra terms in the dynamics and cost. To provide a closed-form solution to \eqref{eq:iterlinear}, we next extend the results in Section \ref{sec:linearcov} to the following variation of linear covariance steering problems with a drift term in the dynamics and an extra linear term in the cost
	\begin{subequations}\label{eq:covcontrollineardrift}
	\begin{eqnarray}
		\min_{u} \!\!\!&&\!\!\!\! \mE \left\{\int_0^1 [\frac{1}{2}\|u_t\|^2 + \frac{1}{2} X_t^T Q(t) X_t + X_t^T r(t)]dt\right\}
		\\\!\!\!&&\!\!\!\! dX_t = A(t)X_t dt +a(t)dt+\! B(t)(u_tdt+ \!\sqrt{\epsilon} dW_t)
		\\\!\!\!&&\!\!\!\! X_0 \sim \cN(m_0, \Sigma_0),\quad X_1 \sim \cN(m_1, \Sigma_1).
	\end{eqnarray}
	\end{subequations}
By linearity, the mean and the covariance can be controlled separately. Apparently, the covariance control part is the same as Problem \eqref{eq:covcontrollinear} since both the dynamics of covariance and the cost related to covariance in \eqref{eq:covcontrollineardrift} are exactly the same as \eqref{eq:covcontrollinear}.

The deterministic control for the mean is slightly different from \eqref{eq:covcontrollinear}. Let $v_t$ be the mean of the control, then it is associated with the deterministic optimal control problem
	\begin{subequations}\label{eq:deterministiccontrol}
	\begin{eqnarray}
		\min_{v} && \int_0^1 [\frac{1}{2}\|v_t\|^2 + \frac{1}{2} x_t^T Q(t) x_t+x_t^T r(t)]dt
		\\&& \dot x_t = A(t)x_t + a(t)+B(t)v_t
		\\&& x_0 =m_0, \quad x_1 = m_1.
	\end{eqnarray}
	\end{subequations}
The above problem can be solved using the Pontryagin's principle. More specifically, it amounts to solving the differential equations
	\begin{equation}\label{eq:xlambda}
		\left[\begin{matrix} \dot x_t \\ \dot \lambda_t\end{matrix}\right]
		=
		\left[\begin{matrix} A(t) & -B(t)B(t)^T\\ -Q(t) & -A(t)^T\end{matrix}\right]
		\left[\begin{matrix}  x_t \\ \lambda_t\end{matrix}\right]
		+
		\left[\begin{matrix} a(t) \\ -r(t)\end{matrix}\right]
	\end{equation}
with boundary condition $x_0 = m_0, x_1 = m_1$. Using the notation in Section \ref{sec:back} we obtain
	\begin{equation*}
		\left[\begin{matrix}  x_1 \\ \lambda_1\end{matrix}\right]\!=\!\!
		\left[
        \begin{matrix}
        \Phi_{11} & \Phi_{12}\\
        \Phi_{21} & \Phi_{22}
        \end{matrix}\right]\!
        \left[\begin{matrix}  x_0 \\ \lambda_0\end{matrix}\right]
        		\!+\!\int_0^1 \left[\begin{matrix}
        \Phi_{11}(1,\tau) \!\!\!&\!\! \Phi_{12}(1,\tau)\\
        \Phi_{21}(1,\tau) \!\!\!&\!\! \Phi_{22}(1,\tau)
        \end{matrix}\right]\left[\begin{matrix} a(\tau) \\ -r(\tau)\end{matrix}\right]d\tau.
	\end{equation*}
It follows
	\begin{equation*}
		\lambda_0 = \Phi_{12}^{-1}\left(m_1\!\!-\!\Phi_{11}m_0\!\!-\!\!\int_0^1 (\Phi_{11}(1,\tau)a(\tau)\!\! -\!\Phi_{12}(1,\tau)r(\tau))d\tau\right).
	\end{equation*}
Plugging it back to \eqref{eq:xlambda} gives the expression for $\lambda_t$. The solution to the deterministic control problem \eqref{eq:deterministiccontrol} is $v_t^\star=-B(t)^T \lambda_t$. The optimal control strategy for the covariance steering problem \eqref{eq:covcontrollineardrift} is
	\[
		u_t^\star = -B(t)^T\Pi(t) (X_t-x_t^\star) + v_t^\star.
	\] 

\subsection{Construct optimal control}\label{sec:control}
Upon convergence, the total feedback control policy can be recovered as follows. Denote $A^\star, a^\star$ the limit point of $\{A_k, a_k\}$. 
Let $z_t^\star, \Sigma_t^\star$ be the solutions to
	\begin{eqnarray*}
		\dot z_t^\star &=& A^\star(t) z_t^\star + a^\star(t), 
		\\
		\dot \Sigma_t^\star &=& A^\star(t)\Sigma_t^\star +\Sigma_t^\star A^\star(t)^T + \epsilon B(t)B(t)^T,
	\end{eqnarray*}
then they satisfy that $z_0^\star = m_0, z_1^\star = m_1$ and $\Sigma_0^\star = \Sigma_0, \Sigma_1^\star = \Sigma_1$. Denote the linearization of the prior dynamics \eqref{eq:nonlinearzero} along $z_t^\star$ as
	\begin{equation}\label{eq:GM}
		dX_t = \hat A^\star(t) X_t dt + \hat a^\star(t) dt + \sqrt{\epsilon} B(t) dW_t,
	\end{equation}
then
	\begin{subequations}\label{eq:zS}
	\begin{eqnarray}
		\hat A^\star(t) &=& \nabla f(t, z^\star_t)^T
		\\
		\hat a^\star(t) &=& f(t,z^\star_t) - \nabla f(t, z^\star_t)^T z^\star_t.
	\end{eqnarray}
	\end{subequations}

The distribution $\cP^\star$ over the path space $\Omega$ associated with 
	\begin{equation}\label{eq:Pstar}
		dX_t = A^\star(t) X_t dt + a^\star(t) dt + \sqrt{\epsilon} B(t) dW_t
	\end{equation}
is the solution to our covariance steering problem \eqref{eq:noncovcontrol}. To retrieve the optimal control strategy, 
consider the optimization 
	\begin{equation}\label{eq:finalpro}
		\min_{\cP\in\hat\Pi(\rho_0, \rho_1)} G(\cP) + \langle \frac{\delta F}{\delta \cP}(\cP^\star), \cP\rangle. 
	\end{equation}
It shares the same minimizer, which is $\cP^\star$, as 
	\begin{equation}\label{eq:iter}
		\min_{\cP\in\hat\Pi(\rho_0, \rho_1)} G(\cP) + \langle \frac{\delta F}{\delta \cP}(\cP^\star), \cP\rangle + \frac{1}{\eta} {\rm KL}( \cP\| \cP^\star).
	\end{equation}
Indeed, since $\cP^\star$ is the minimizer of \eqref{eq:iter} and ${\rm KL}( \cP\| \cP^\star)$ vanishes when $\cP = \cP^\star$, $\cP^\star$ must be a minimizer of \eqref{eq:finalpro} as well. To retrieve the final optimal control from \eqref{eq:finalpro}, we reformulate it as a linear covariance control problem as follows. The proof is similar to that of Theorem \ref{thm:eachiteration} and is thus omitted. 
\begin{prop}\label{prop:retrieve}
The optimization \eqref{eq:finalpro} amounts to the linear covariance steering problem
	\begin{subequations}
	\begin{eqnarray}
		\min_{u}\!\!\!\! &&\!\!\!\! \mE \left\{\int_0^1 [\frac{1}{2}\|u_t\|^2 + \frac{1}{2} X_t^T Q^\star(t) X_t + X_t^T r^\star(t)]dt\right\}
		\\ \!\!\!\! &&\!\!\!\!  \hspace{-0.2cm}dX_t = \hat A^\star(t)X_t dt + \hat a^\star(t)dt +\! B(t)(u_tdt\!+\! \sqrt{\epsilon} dW_t)
		\\ && X_0 \sim \rho_0,\quad X_1 \sim \rho_1,
	\end{eqnarray}
	\end{subequations}
where $Q^\star(t) = \nabla^2 V(z^\star_t)$
and
	\begin{align*}
		r^\star(t) =& \nabla V(z^\star_t)+\frac{1}{2}[\nabla \tr(\nabla^2V(z^\star_t)\Sigma^\star_t)-2\nabla^2V(z^\star_t)z^\star_t
		\\&\hspace{-0.5cm}+\nabla\tr( (B(t)B(t)^T)^{\dagger}(\nabla f^T -A^\star)\Sigma^\star_t(\nabla f-A^{\star T}))].
	\end{align*}
\end{prop}
	
Denote by 
	\begin{equation}\label{eq:totalcontrol}
		u_t^\star = K^\star(t) X_t + d^\star(t)
	\end{equation}
the optimal control for the linear covariance control problem in Proposition \ref{prop:retrieve}, then it is also the optimal control strategy to the nonlinear covariance steering problem \eqref{eq:noncovcontrol}.

\section{Generalizations}\label{sec:extension}
The nonlinear covariance steering algorithm can be extended to the general control-affine system 
	\begin{equation}\label{eq:nonlineargeneral}
		dX_t = f(t,X_t)dt + g(t,x) (u_t dt + \sqrt{\epsilon} d W_t).
	\end{equation}
The goal is again to approximately solve the density control problem
	\begin{subequations}\label{eq:covcontrolN}
	\begin{eqnarray}\label{eq:covcontrolN1}
		\min_{u_t} && \mE \left\{\int_0^1 [\frac{1}{2}\|u_t\|^2 + V(X_t)]dt\right\}
		\\\label{eq:covcontrolN2}
		&&dX_t = f(t,X_t)dt + g(t,X_t) (u_t dt + \sqrt{\epsilon} d W_t)
		\\\label{eq:covcontrolN3}
		&& 
		X_0 \sim \cN(m_0,\Sigma_0),\quad X_1 \sim \cN(m_1,\Sigma_1).
	\end{eqnarray}
	\end{subequations}

As before, the above can be reformulated as an optimization \eqref{eq:noncovcontrol} over $\cP^u$ and can be solved using the proximal gradient algorithm with iteration \eqref{eq:iteration}. Again, each iteration of \eqref{eq:iteration} is a covariance control problem for linear dynamics. However, the explicit realization of each iteration is slightly different. 

The probability measure $\cP^u_k$ is a Gaussian Markov process
	\begin{equation}\label{eq:GMk}
		dX_t = A_k(t) X_t dt + a_k(t) dt + \sqrt{\epsilon} g(t,z_t^k) dW_t.
	\end{equation}
where the mean $z^k$ and the covariance $\Sigma^k$ satisfy
	\begin{subequations}\label{eq:zSk}
	\begin{eqnarray}
		\dot z^k_t &=& A_k(t)z^k_t +a_k(t)
		\\
		\dot \Sigma_t^k &=& A_k(t)\Sigma_t^k +\Sigma_t^k A_k(t)^T + \epsilon g(t,z_t^k)g(t,z_t^k)^T
	\end{eqnarray}
	\end{subequations}
The above equations \eqref{eq:GMk}-\eqref{eq:zSk} should be compared with \eqref{eq:GM}-\eqref{eq:zS}. The difference between them are due to the fact $g(t,x)$ depends on the value of state. In \eqref{eq:GMk}-\eqref{eq:zSk}, we use the value of $g$ at the mean $z_t^k$ as an approximation.

Following the same analysis as in the Section \ref{sec:algorithm}, we obtain the explicit updates
	\begin{subequations}\label{eq:updateAaN}
	\begin{eqnarray}\label{eq:updateAaN1}
		A_{k+1}(t) &=&  \frac{1}{1+\eta}[A_k(t) + \eta \hat A_k(t)] + g(t,z_t^k)K_k(t)
		\\\label{eq:updateAaN2}
		a_{k+1}(t) &=& \frac{1}{1+\eta}[a_k(t)+ \eta \hat a_k(t)] + g(t,z_t^k)d_k(t).
	\end{eqnarray}
	\end{subequations}

Once $\{A_k, a_k\}$ converges to $(A^\star, a^\star)$, we can follow exactly the same steps as in Section \ref{sec:control} (after replacing $B(t)$ by $g(t, z^\star_t)$) to retrieve the final optimal affine controller. 


\section{Numerical examples}\label{sec:example}
{\red
In this section we present several examples to illustrate our algorithm. All the experiments are carried out using a PC with 64GB RAM and an Intel Core i9 CPU.

\subsection{Double integrator with drag}
We consider the same nonlinear dynamical system used in \cite{RidOkaTsi19}
    \begin{subequations}\label{eq:double_integrator}
	\begin{align}   
			d x_1&=x_2 d t,\\
			d x_2&=(u - c_d \lVert x_2 \rVert) d t+ \sqrt{\epsilon} d W_t,
	\end{align}
\end{subequations}
where $c_d$ represents drag coefficient. To compare the proposed algorithm with the \textit{iterative Covariance Steering (iCS)} algorithm \cite{RidOkaTsi19}, we conduct experiments for the two algorithms under the same settings but removing the chance constraints and trust region constraints in iCS. We choose $c_d=0.005$ and $\epsilon=0.1$, and also apply a time scaling to map time interval $[0,1]$ to $[0, T]$. We choose the initial mean and covariance conditions as $m_0 = \left[\begin{matrix}1, 8, 2, 0\end{matrix}\right]^T, \quad \Sigma_0 = 0.01 I$
and the target mean and covariance is chosen to be $m_T = \left[\begin{matrix}1, 2, -1, 0\end{matrix}\right]^T, \Sigma_T = 0.1 I$.
In addition to control energy, in this experiment we add also a quadratic state cost $V(x_t)=(x_t-z^k_t)^TQ_t(x_t-z^k_t), Q_t=0.1I$ in \eqref{eq:covcontrol} as in \cite{RidOkaTsi19}. The stopping criteria for the experiments is when the average over all time steps of relative Frobenius norm error of two consecutive $A_k$ and $a_k$ is less than a predefined threshold. Fig. \ref{fig:double_integ_pos_vel} shows the resulting position and velocity distributions and sampled trajectories. We display the convergence time for different time discretizations in Tab. \ref{tab:compare_iCS}. In the same settings, our proposed algorithm converges much faster (more than 1000 times speedup).
We note that both algorithms converge to a locally minimum solution, and the reason for the convergence rate difference lies in that the baseline algorithm relies on optimization over discrete dynamics which leads to a constrained semi-definite program at each iteration while our proposed method optimizes directly over continuous dynamics at each iteration which has a closed-form solution.
Our algorithm also has the advantage that its complexity increases linearly with time discretization.
In terms of stability, our algorithm is guaranteed to have a solution that strictly satisfies the boundary conditions at each iteration if the linearized system is controllable. 

\begin{table}[]
\centering
\begin{tabular}{|c|cc|cc|}
\hline
    & \multicolumn{2}{c|}{iCS} & \multicolumn{2}{c|}{ours} \\ 
\hline
Time discretization &   \multicolumn{1}{c|}{25} & 50     & \multicolumn{1}{c|}{25} & \textbf{1000} \\ 
\hline
Solving time   & \multicolumn{1}{c|}{126.1212} & 1528.1  & \multicolumn{1}{c|}{\textbf{0.1184}} & \textbf{0.2741} \\ 
\hline
Cost   & \multicolumn{1}{c|}{31.3198} & 30.7929  & \multicolumn{1}{c|}{2.9945} & 2.9283 \\ 
\hline
\end{tabular}
\caption{Convergence time and cost comparison between the proposed method and the iCS algorithm proposed in \cite{RidOkaTsi19} for the dynamics \eqref{eq:double_integrator}. The costs are computed using optimal controlled dynamics obtained from the two algorithms using Monte Carlo.
}
    \label{tab:compare_iCS}
\end{table}
\begin{figure}[tb]
    \centering
    \begin{subfigure}[b]{0.24\textwidth}
    \centering
    \includegraphics[width=1\textwidth]{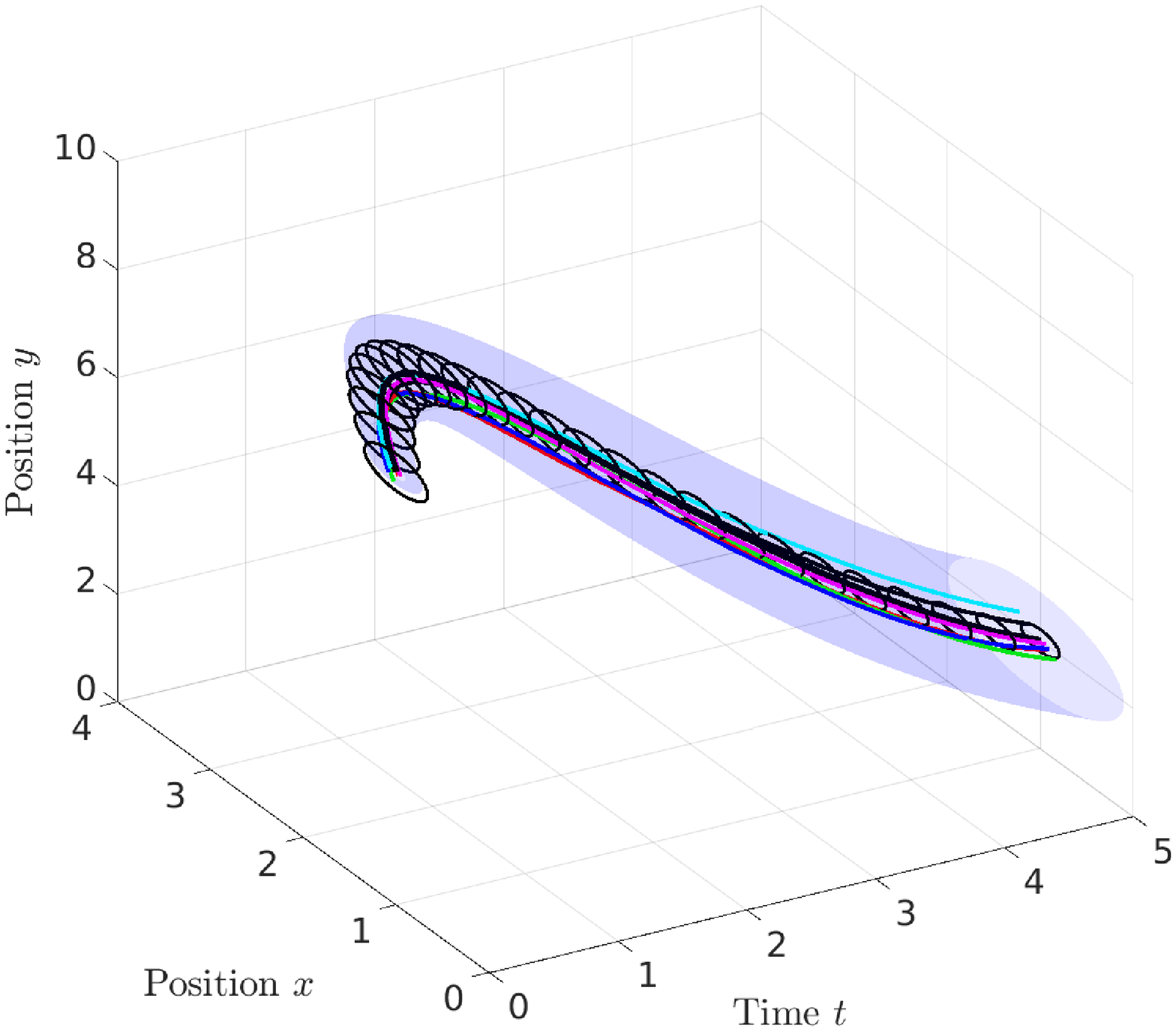}
    \caption{position}
    \label{fig:double_integ_pos_vel_a}
    \end{subfigure}
    \hfill
    \begin{subfigure}[b]{0.24\textwidth}
    \centering
    \includegraphics[width=1\textwidth]{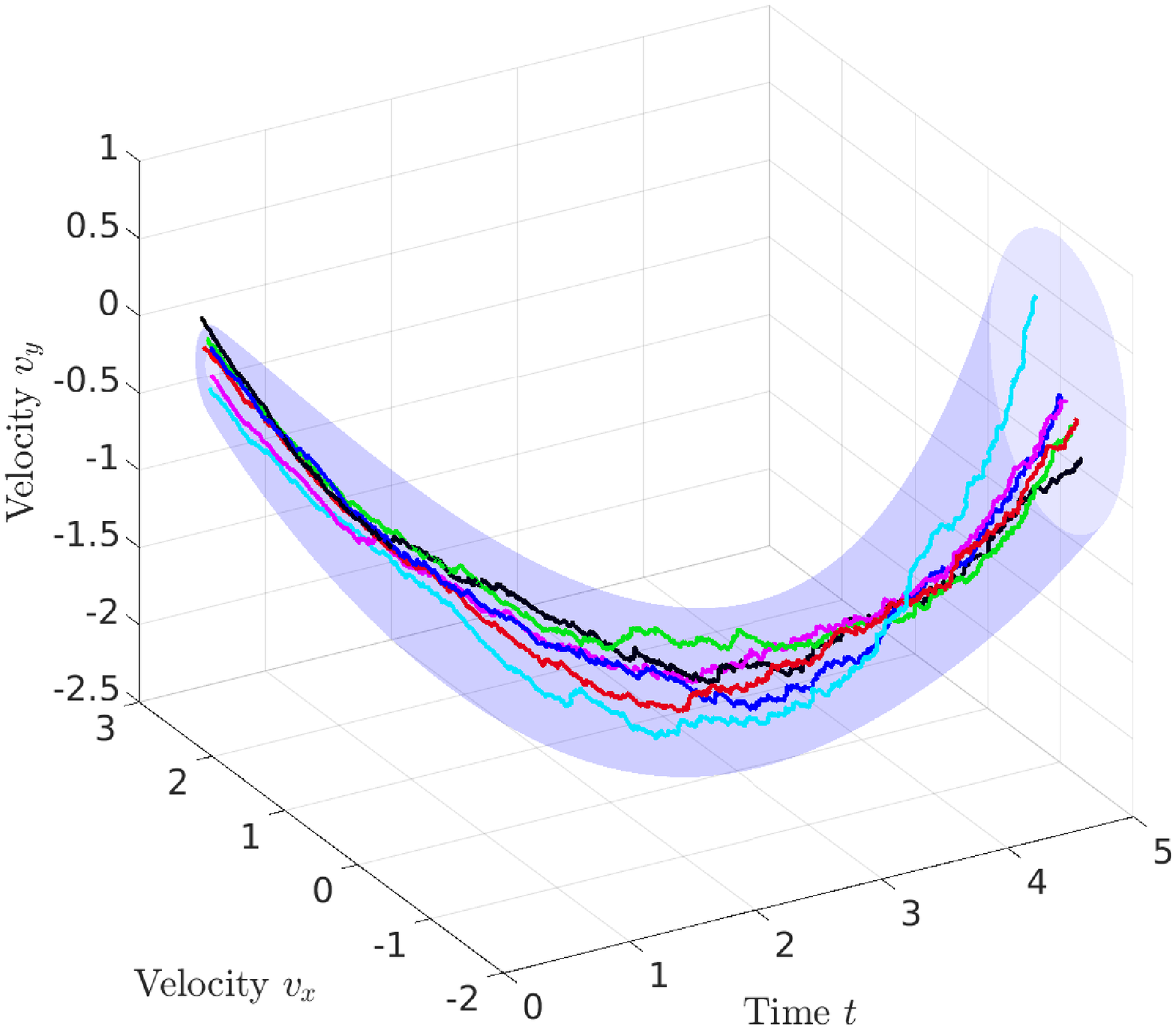}
    \caption{velocity}
         \label{fig:double_integ_pos_vel_b}
    \end{subfigure}
    \caption{Covariance and sampled position and velocity trajectories plot of the double integrator with drag dynamics \eqref{eq:double_integrator}, with a time scaling to $T = 5$ and time discretization $N=1000$. We plot the covariances obtained from iCS for $N=25$ in Fig. \ref{fig:double_integ_pos_vel_a} in black ellipsoids. In contrast to iCS \cite{RidOkaTsi19}, our result satisfies the final time constraints.}
    \label{fig:double_integ_pos_vel}
\end{figure}


\subsection{3-link manipulator}
We next consider a 3-link manipulator example. By this example, we show that the proposed algorithm is effective in controlling a system governed by a manipulator equation in the form
\begin{equation}
    M(q)\ddot{q} + C(q,\dot q) + g(q) = \tau+\sqrt{\epsilon}dW_t,
\end{equation}
where $q=[\theta_1, \theta_2, \theta_3]^T$ is the joint angles, $M(q)$ is the mass matrix, $C(q,\dot q)$ is the Coriolis matrix, $g(q)$ is the gravity term and $\tau$ is the torque input applied on joints. 
Denote the state of the system as
$X = [q,\dot q]^T$ and the input as $u=\tau$, 
then we have the nonlinear control-affine system representation for the manipulator following (8)
\begin{equation}\label{eq:3dof_eom}
    dX_t =
    \left[\begin{matrix}
    \dot q \\ M^{-1}(q)(-C(q, \dot q)\!-\!g(q))
    \end{matrix}\right]dt
    +\!
    \left[\begin{matrix}
    0 \\ M^{-1}(q)(u_tdt\!+\!\sqrt{\epsilon}dW_t)
    \end{matrix}\right]
\end{equation}

We choose the number of discretizations to 1000 with a time span $T=1$ and noise intensity $\epsilon=0.1$. We steer the system from initial mean $m_0=[0,0,0,0,0,0]^T$ and covariance $\Sigma_0=0.05I$, to $m_T=[-\frac{\pi}{2},\frac{\pi}{2},\frac{\pi}{2},0,0,0]^T$ and covariance $\Sigma_T=0.01I$.
With convergence error set to $10^{-5}$, it takes around $41$ iterations and $6.2s$ on average to converge.
Considering the underlying complexity of dynamics and the number of discretizations, the experiment shows the proposed algorithm is effective in solving the covariance steering problem for manipulator systems. The mean and sampled trajectories of the nonlinear stochastic system \eqref{eq:3dof_eom} under optimal feedback control policy are shown in Fig. \ref{fig:3dof_velpos}. To illustrate the evolution of covariance and simulated trajectory, we show the covariance ellipsoid of joint angles in Fig. \ref{fig:3dof_funnel}. The initial covariance shrinks as time goes on and reaches the target covariance. The simulated trajectories started in randomly sampled states from the initial distribution are bounded inside the $3\sigma$-confident ellipsoid at each time stamp $t$. The trajectory simulation for 3-DOF manipulator is shown in Fig. \ref{fig:3dof_motion}.

\begin{figure}[tb]
    \centering
    \begin{subfigure}[b]{0.24\textwidth}
    \centering
    \includegraphics[width=1\textwidth]{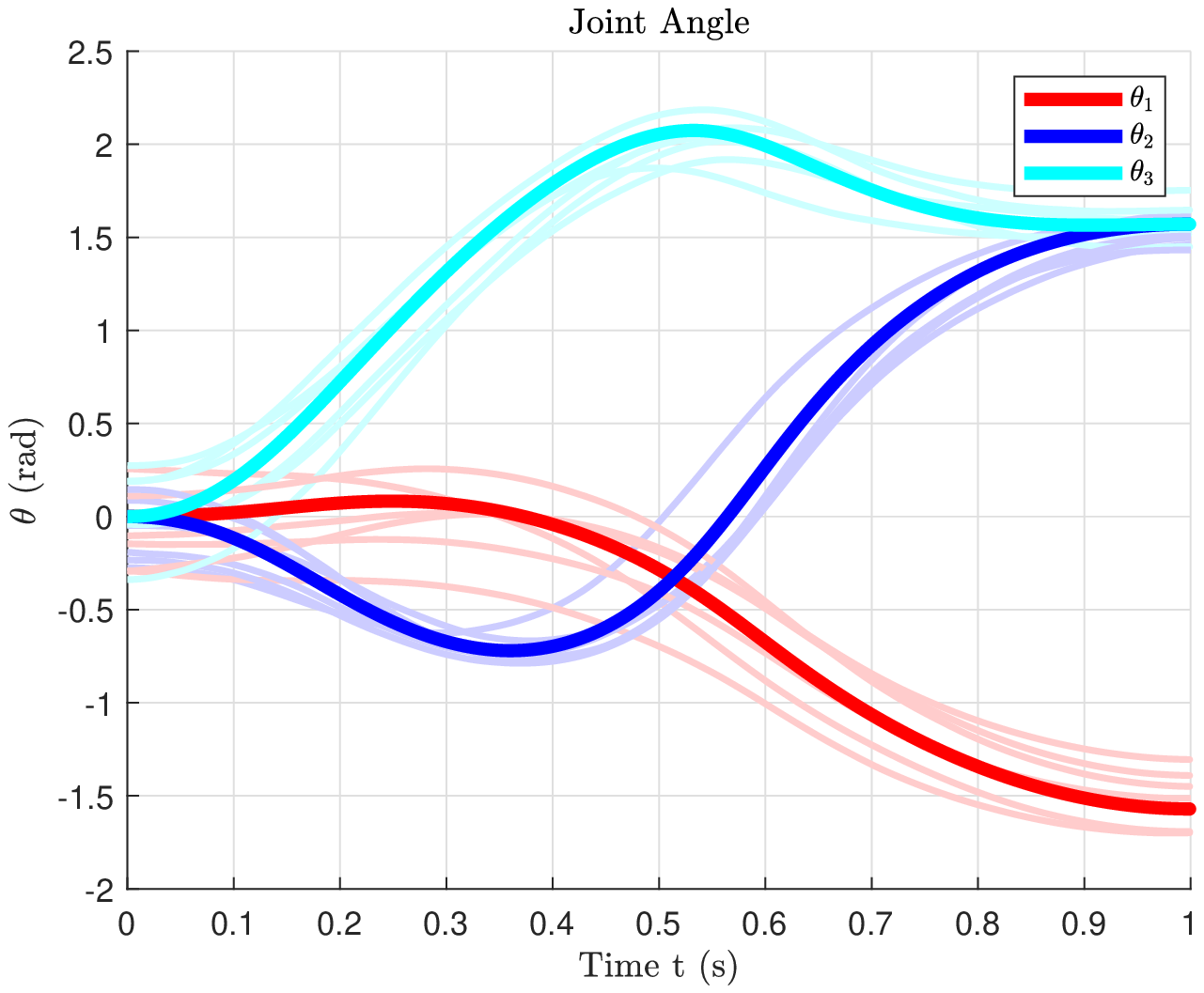}
    \caption{position}
    \label{fig:3dof_velpos_a}
    \end{subfigure}
    \hfill
    \begin{subfigure}[b]{0.24\textwidth}
    \centering
    \includegraphics[width=1\textwidth]{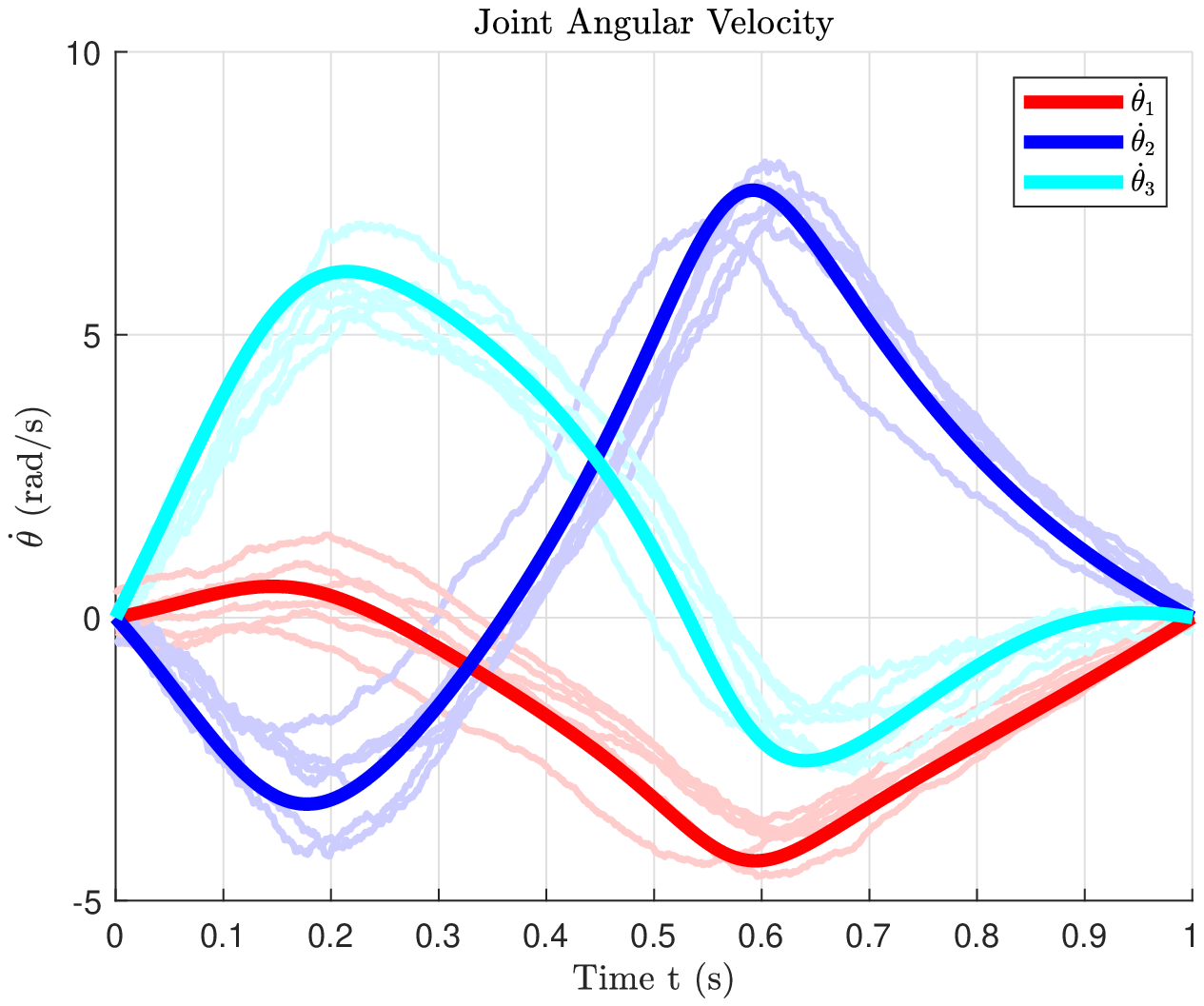}
    \caption{velocity}
         \label{fig:3dof_velpos_b}
    \end{subfigure}
    \caption{Joint angle and velocity trajectory plot of the 3-DOF robot arm. The trajectories are simulated with full dynamics and noise injected from torque inputs. Mean trajectories of joint angles and velocities are highlighted.}
    \label{fig:3dof_velpos}
\end{figure}

\begin{figure}
    \centering
    \includegraphics[width=0.4\textwidth]{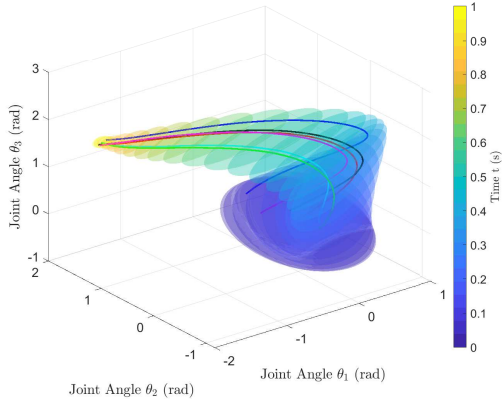}
    \caption{$3\sigma$ covariance ellipsoid trajectory in joint space. Each joint angle  ellipsoid is retrieved from covariance matrix $\Sigma_t^{k}$ which satisfies (51b). The position of the ellipsoid is the mean $z_t^k$ at time t. 
    }
    \label{fig:3dof_funnel}
\end{figure}

\begin{figure}
    \centering
    \includegraphics[width=0.4\textwidth]{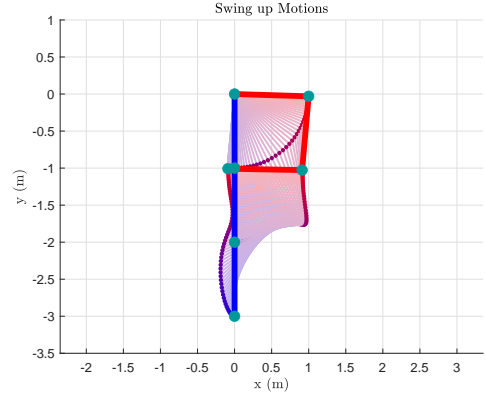}
    \caption{Simulated trajectory of 3-DOF manipulator. The starting position is $X_0=[0,0,0,0,0,0]^T$ (hanging vertically, shown in blue), and the generated control policy steers the system to the vicinity of target state $X_T=[-\frac{\pi}{2},\frac{\pi}{2},\frac{\pi}{2},0,0,0]^T$ (shown in red). The final position is not perfectly aligned with the target due to the stochasticity of the system.} 
    \label{fig:3dof_motion}
\end{figure}
}

\section{Conclusion}\label{sec:conclusion}
In this paper we presented a new approach to covariance steering problems for control-affine systems. We showed that this nonlinear covariance control problem can be reformulated as an optimization over the space of probability distributions of the trajectories. We then developed an efficient algorithm based on the proximal gradient algorithm. Each proximal iteration in the algorithm requires solving a linear covariance control problems, whose solutions exist in closed form. One distinguishing feature of this distributional formulation of nonlinear covariance steering and the resulting algorithm is that the optimal control is obtained for the continuous-time problems directly and the complexity of the algorithm is insensitive to the resolution of time-discretization. One potential future research direction is to extend the framework to account for hybrid dynamics. 

\appendix 

\subsection{Proof of Lemma \ref{lem:variationP}}
By definition \eqref{eq:FP}, the first order expansion of $F(\cP+ \delta \cP) - F(\cP)$ is
	\begin{eqnarray}\label{eq:variationF}
		&& \langle \frac{1}{\epsilon}\hat V (\cP+\delta \cP)-\log d\hat \cP^0(\cP+\delta \cP), \cP+\delta \cP\rangle
		\\&&-  \langle \frac{1}{\epsilon}\hat V (\cP)-\log d\hat \cP^0(\cP), \cP\rangle \nonumber
		\\\nonumber &\approx&\langle \frac{1}{\epsilon}\hat V (\cP)-\log d\hat \cP^0(\cP),\delta \cP\rangle 
		\\&&+\langle \frac{1}{\epsilon}\hat V (\cP+\delta \cP)-\frac{1}{\epsilon}\hat V (\cP)-\log \frac{d\hat \cP^0(\cP+\delta \cP)}{d\hat \cP^0(\cP)}, \cP\rangle. \nonumber
	\end{eqnarray}
Here we use the notations $\hat V(\cP), \hat\cP^0(\cP)$ to emphasize the dependence of $\hat V, \hat\cP^0$ on $\cP$. 
By \eqref{eq:linearizeV}, $\langle \hat V (\cP+\delta \cP)-\hat V (\cP), \cP\rangle$ is approximately
	\begin{eqnarray*}
		&& \langle V(z_t+\delta z_t)-V(z_t)-\delta z_t^T \nabla V(z_t) ,\cP\rangle
		\\&+& \langle \frac{1}{2}(x^T-z_t^T-\delta z_t^T) \nabla^2 V(z_t+\delta z_t) (x-z_t-\delta z_t)
		\\&-&\frac{1}{2}(x^T-z_t^T) \nabla^2 V(z_t) (x-z_t),\cP\rangle.
	\end{eqnarray*}
We only keep first order approximations {\red and have used the fact that $z_t$ is the mean of $\cP$.} The first term is clearly $0$ after ignoring high order information. The second term is approximately
{\red
	\begin{eqnarray*}
		&&\langle \frac{1}{2}(x^T-z_t^T) \nabla^2 V(z_t+\delta z_t) (x-z_t)
		\\&&-\frac{1}{2}(x^T-z_t^T) \nabla^2 V(z_t) (x-z_t),\cP\rangle
             \\ =&& \frac{1}{2} \int_0^1[\tr (\nabla^2 V(z_t+\delta z_t) \Sigma_t)-\tr (\nabla^2 V(z_t) \Sigma_t)] dt
		\\\approx&& \frac{1}{2}\int_0^1[\delta z_t^T \nabla\tr(\nabla^2 V(z_t) \Sigma_t)] dt.
	\end{eqnarray*}
 }
It follows that 
	\begin{equation}\label{eq:Vhat}
		\langle \hat V (\cP+\delta \cP)-\hat V (\cP), \cP\rangle \approx \frac{1}{2}\langle x^T \nabla\tr(\nabla^2 V(z_t) \Sigma_t), \delta \cP\rangle.
	\end{equation}
 {\red
By Girsanov theorem \eqref{eq:girsanov}, 
	\[
		\frac{d \cP}{d\hat \cP^0(\cP)} = \exp \left[\int_0^1 \frac{1}{2\epsilon} \|U\|^2_{(BB^T)^\dagger} dt + \frac{1}{\sqrt{\epsilon}}\langle U, B dW_t\rangle_{(BB^T)^\dagger} \right]
	\]
where 
	\[
		U = AX_t +a -\nabla f(z_t)^T X_t -f(z_t)+\nabla f(z_t)^T z_t.
	\]
A similar expression can be obtained for $d\cP / d\hat \cP^0(\cP+\delta \cP)$. Combining them and using first order approximation yields
	\begin{eqnarray}\label{eq:P0hat}
		&&\langle\log \frac{d\hat \cP^0(\cP)}{d\hat \cP^0(\cP+\delta \cP)}, \cP\rangle
		\\\approx&& \langle \frac{1}{2\epsilon} x^T \nabla \tr((BB^T)^{\dagger}(\nabla f(z_t)^T- A) \Sigma_t (\nabla f(z_t)-A^T)), \delta \cP\rangle. \nonumber
	\end{eqnarray}
Plugging \eqref{eq:Vhat} and \eqref{eq:P0hat} into \eqref{eq:variationF} yields \eqref{eq:gradietnF}. 
}

{
\bibliographystyle{IEEEtran}
\bibliography{./refs}
}
\end{document}